\documentclass[a4paper]{amsart}
\usepackage{multicol,ifthen}
\usepackage{amsmath}
\usepackage{amssymb}
\usepackage{amsthm} 
\usepackage{enumerate}

\begin{document}

\newcommand{\F}{\mathcal{F}}
\newcommand{\Sc}{\mathcal{S}}
\newcommand{\R}{\mathbb R}
\newcommand{\T}{\mathbb T}
\newcommand{\N}{\mathbb N}
\newcommand{\Z}{\mathbb Z}
\newcommand{\C}{\mathbb C}  
\newcommand{\h}[2]{\mbox{$ \widehat{H}^{#1}_{#2}(\R)$}}
\newcommand{\hh}[3]{\mbox{$ \widehat{H}^{#1}_{#2, #3}$}} 
\newcommand{\n}[2]{\mbox{$ \| #1\| _{ #2} $}} 
\newcommand{\x}{\mbox{$X^r_{s,b}$}} 
\newcommand{\xx}{\mbox{$X_{s,b}$}}
\newcommand{\X}[3]{\mbox{$X^{#1}_{#2,#3}$}} 
\newcommand{\XX}[2]{\mbox{$X_{#1,#2}$}}
\newcommand{\q}[2]{\mbox{$ {\| #1 \|}^2_{#2} $}}
\newcommand{\e}{\varepsilon}
\newcommand{\om}{\omega}
\newcommand{\lb}{\langle}
\newcommand{\rb}{\rangle}
\newcommand{\ls}{\lesssim}
\newcommand{\gs}{\gtrsim}
\newcommand{\pd}{\partial}
\newtheorem{lemma}{Lemma} 
\newtheorem{kor}{Corollary} 
\newtheorem{theorem}{Theorem}
\newtheorem{prop}{Proposition}

\title[gZK at critical regularity]{On the generalized Zakharov-Kuznetsov equation at critical regularity}

\author[Axel~Gr{\"u}nrock]{Axel~Gr{\"u}nrock}

\address{Axel~Gr{\"u}nrock: Heinrich-Heine-Universit\"at D\"usseldorf,
Mathematisches Institut, Universit\"atsstrasse 1, 40225 D\"usseldorf, Germany.}
\email{gruenroc@math.uni-duesseldorf.de}

\subjclass[2010]{Primary: 35Q53. Secondary: 37K40}

\begin{abstract} The Cauchy problem for the generalized Zakharov-Kuznetsov equation
$$\partial_t u +\partial_x\Delta u=\partial_x u^{k+1}, \qquad \qquad u(0)=u_0$$
\hfill\\
\noindent is considered in space dimensions $n=2$ and $n=3$ for integer exponents $k \ge 3$. For data $u_0 \in \dot{B}^{s_c}_{2,q}$,
where $1\le q \le \infty$ and $s_c=\frac{n}{2}- \frac{2}{k}$ is the critical Sobolev regularity, it is shown, that this problem
is locally well-posed and globally well-posed, if the data are sufficiently small. The proof follows ideas of Kenig, Ponce, and
Vega \cite{KPV93} and uses estimates for the corresponding linear equation, such as local smoothing effect, Strichartz estimates,
and maximal function inequalities. These are inserted into the framework of the function spaces $U^p$ and $V^p$ introduced by
Koch and Tataru \cite{KT05}, \cite{KT07}.
\end{abstract}

\keywords{generalized Zakharov-Kuznetsov equation -- local and global well-posedness -- critical regularity}

\maketitle
\tableofcontents

\section{Introduction}

The Zakharov-Kuznetsov equation (ZK)
\begin{equation}\label{ZK}
 \partial_t u +\partial_x\Delta u=\partial_x u^{2}
\end{equation}
with $\Delta= \partial_x^2 + \sum_{i=1}^{n-1}\partial^2_{y_i}$, $(x,y)\in \R\times\R^{n-1}$, is a generalization of the famous
Korteweg-de Vries equation (KdV) to arbitrary higher dimensions. In 1974, Zakharov and Kuznetsov derived \eqref{ZK} as a model
describing the unidirectional wave propagation in a magnetized plasma in three space dimensions \cite[equation (6)]{ZK74}. For
two dimensions, a derivation of \eqref{ZK} from the basic hydrodynamic equations is due to Laedke and Spatschek \cite[Appendix B]{LaSp}.
We also refer to the paper \cite{LLS} by Lannes, Linares, and Saut for a rigorous justification of ZK valid for $n \in \{2,3\}$.
Both, the Cauchy problem as well as several initial boundary value problems connected with \eqref{ZK} have attracted considerable
interest in recent years, we mention \cite{BL}, \cite{BJM}, \cite{F95}, \cite{GH}, \cite{LS}, \cite{MP15}, \cite{RV12}; \cite{F08}, \cite{LaTr},
\cite{LPS}, \cite{STW}. This list is by no means exhaustive. Similar as for KdV and - to the author's knowledge -
beginning with the work \cite{BL} of Biagioni and Linares on the modified equation, generalizations of ZK with higher power nonlinearities
\begin{equation}\label{gZK}
 \partial_t u +\partial_x\Delta u=\partial_x u^{k+1} \qquad \qquad \mbox{with} \qquad \qquad u(0)=u_0
\end{equation}
are considered, too. We call \eqref{gZK} the k-th generalized ZK equation, for short gZK-k. In this paper we are concerned with local
and small data global well-posedness of the Cauchy problem for gZK-k in two and three space dimensions for integers $k\ge3$. (Unfortunately
our arguments break down for the modified equation, i. e. for $k=2$.) In the $2$ $D$ - case the following results are known for data in
the classical Soboles spaces $H^s$.
\begin{itemize}
 \item In 2011 Linares and Pastor \cite{LP11} showed that the Cauchy problem for gZK-k is locally well-posed in $H^s$, if $k\ge2$
and $s>\max(\frac34,1-\frac{3}{2k-4})$. If the data are sufficiently small in $H^1$, then the corresponding solutions extend globally in time.
 \item For $k>8$ the lower bound on $s$ was pushed down to $s>1-\frac{2}{k}$ by Farah, Linares, and Pastor \cite{FLP} in 2012. Since
$s_c=1-\frac{2}{k}$ is the critical regularity by scaling considerations, this result covers the whole subcritical range.
 \item Further progress on the local problem was reached by Ribaud and Vento \cite{RV12a} in 2012. Their results almost reached $s_c$
for all $k\ge4$, while for the quartic nonlinearity they assume $s>\frac{5}{12}$.
\end{itemize}
Further results on gZK-k in two dimensions with data in weighted spaces were recently obtained by Fonseca and Pachon \cite{FP}. The author
is not aware of any comparable results for $k\ge3$ in the three dimensional case, where the critical regularity is $s_c=\frac32-\frac{2}{k}$.
More generally we have
$$s_c=s_c(n,k)=\frac{n}{2}-\frac{2}{k}.$$
Roughly speaking, the method of proof is the same in all three papers \cite{LP11}, \cite{FLP}, and \cite{RV12a}. The authors adapt
the strategy developped by Kenig, Ponce, and Vega in \cite{KPV93} in the KdV-context and apply a combination of local smoothing
estimate, Strichartz inequality, and maximal function estimate in a contraction mapping argument. Here we shall pick up these ideas,
push them down to the critical regularity and extend the arguments to the three dimensional case. Following Molinet-Ribaud \cite{MR03}
and especially in our method of proof Koch-Marzuola \cite{KoMa} in their works on gKdV, we consider data in the homogeneous Besov spaces
$$\dot{B}^s_{2,q}=\{u_0 \in {\mathcal{Z}}': \|u_0\|_{\dot{B}^s_{2,q}}\},$$
where ${\mathcal{Z}}'$ is the dual space of
$$\mathcal{Z} = \{ f \in \Sc : (D^{\alpha}\F f) (0)=0 \mbox{ for every multi-index } \alpha \}.$$
Here and below $\F$ denotes the Fourier transform. For $q<\infty$ the Besov-norm is in general given by 
$$\|u_0\|_{\dot{B}^s_{p,q}}=\Big(\sum_{N\in 2^{\Z}} \|P_N u_0\|_{L^p}\Big)^{\frac{1}{q}},$$
$P_N=\F^{-1}\chi_{\{|\xi|\sim N\}}\F$ are the Littlewood-Paley projections. A case of special interest is $q=2$, where
$\dot{B}^s_{2,2}=\dot{H}^s$, the homogeneous Sobolev (or Riesz-potential) space. For $q=\infty$ one has the usual modification
$\|u_0\|_{\dot{B}^s_{p,\infty}}=\sup_{N\in 2^{\Z}} \|P_N u_0\|_{L^p}$, and in this case (with $p=2$) we will in addition assume for our
data, that
$$\lim_{N\to\infty}\|P_N u_0\|_{L^2}=\lim_{N\to0}\|P_N u_0\|_{L^2}=0.$$
With $\dot{B}^{s,o}_{2,\infty}$ we will denote the closed subspace of all $u_0 \in \dot{B}^s_{2,\infty}$, for which these limits vanish.
Then $\mathcal{Z}$ is dense in $\dot{B}^{s,o}_{2,\infty}$. Without this additional assumption, several of our arguments break down, e. g.
we loose the persistence property of the solution. Observe for $1\le \tilde{q}\le 2 \le q < \infty$ the inclusions
$$\dot{B}^{s}_{2,1}\subset \dot{B}^{s}_{2,\tilde{q}} \subset \dot{H}^s \subset \dot{B}^{s}_{2,q} \subset\dot{B}^{s,o}_{2,\infty} \subset \dot{B}^s_{2,\infty},$$
so for fixed $s_c$, on the fine scale of the $q$'s, the $\dot{B}^{s_c,o}_{2,\infty}$ is the largest data space we cope with. After
these preparations we can state our

\quad

{\bf{Main result:}} \emph{Let $n\in\{2,3\}$ and $k\ge 3$ an integer. Then the Cauchy problem \eqref{gZK} is locally well-posed for data in
$\dot{B}^{s_c}_{2,q}$, if $q<\infty$, and in $\dot{B}^{s_c,o}_{2,\infty}$. Moreover, we have global well-posedness for small data in these
spaces.} 

\quad

A more precise statement will follow at the end of Section 2. We remark already, that no smallness assumption is needed for the
local part, but that - as usual in a critical case - the lifespan of the solutions cannot be controlled by the size of the data in their
natural norm. To obtain the result, two main difficulties have to be overcome. The first is to prove a sharp global maximal function
estimate or to find a substitute for this. In $2$ $D$ we can solve this problem by symmetrizing the equation, see Section 3.1, especially
Proposition \ref{max} below, while in $3$ $D$ a surprisingly soft argument allows us to circumvent this obstacle, see Section 5.1. The
second problem is the missing generalized Leibniz rule in higher dimensional mixed Lebesgue spaces of type $L_x^pL_t^q$. This is solved
by using the spaces $U^p(L_x^2)$ and $V^p(L_x^2)$ of $L^2_x$-valued functions of the time variable, which were introduced by Koch and
Tataru in \cite{KT05}, \cite{KT07}, see also the exposition by Hadac, Herr, and Koch in \cite{HHK} and Koch's lecture \cite{Koch}. Since
the norms of these spaces depend on the size of the spatial Fourier transform, the "distribution" of derivatives on various factors can
easily be handled. Some basics about these spaces, as far as needed here, are gathered in Section 2.

\quad

In proving the result, we can restrict ourselves to apply \emph{linear} estimates for free solutions - no bilinear refinement of a Strichartz type
inequality is used. For $k=3$ this is astonishing, if we compare our results here with the theory for gKdV. Using linear estimates only,
Kenig, Ponce, and Vega obtained well-posedness for gKdV-3 in $H^s(\R)$ for $s\ge \frac{1}{12}$. To push this down to the critical regularity,
a bilinear estimate for free solutions is needed, see the result in \cite{G05} by the author, which was later on improved by Tao \cite{TT}
to the endpoint and by Koch-Marzuola \cite{KoMa} to critical Besov spaces. As our calculations show, linear estimates are sufficient in
higher dimensions. Furthermore we remark that for the quartic nonlinearity in $2$ $D$ our result closes a gap of $\frac{1}{12}$ derivatives
between the existing LWP theory and the scaling heuristic.

\quad

{\bf{Acknowledgement:}} The author is indepted to Herbert Koch and Sebastian Herr for numerous explanations about the function spaces
$U^p$ and $V^p$.

\newpage 

\section{Function spaces and precise statement of results}

Here we collect the necessary facts about the function spaces $U^p$ and $V^p$, respectively $U_{\varphi}^p$ and $V^p_{\varphi}$. For proofs
and detailed descriptions we refer to the works \cite{HHK} and \cite{Koch}. We begin with the functions of bounded $p$-variation,
which were (in the real valued case) introduced by Wiener in \cite{Wiener}. Let $I \subset \R$ be an interval and $\mathcal{P}_I$
denote the system of all finite partitions $P=\{t_0 < \dots < t_K\} \subset I$ of $I$. Here $t_K=\infty$ is admitted, if $I$ is
unbounded to the right. For a function $v:I\to L^2_x$ the $p$-variation $\omega_p(I,v)$ is defined by
$$\omega_p(I,v):=\sup_{P\in\mathcal{P}_I} \Big( \sum_{k=1}^K\|v(t_k)-v(t_{k-1})\|_{L^2_x}\Big)^{\frac{1}{p}},$$
which, for $1\le p<\infty$, is a seminorm. Setting
$$\|v\|_{V^p}:=\max(\|v\|_{L^{\infty}_I(L^2_x)}, \omega_p(I,v))$$
we get a norm on the linear space
$$V^p(L^2_x):=\{v:I\to L^2_x:\omega_p(I,v)<\infty \},$$
which thereby becomes a B-space. Functions in $V^p(L^2_x)$ are not necessarily continuous, but one sided limits always exist. The closed
subspace of all right continuous functions in $V^p(L^2_x)$ is denoted by $V_{rc}^p(L^2_x)$. For $1\le p < q < \infty$ the embeddings
$$V^p(L^2_x) \subset V^q(L^2_x) \subset L^{\infty}_I(L^2_x)$$
are continuous. Closely related are the function spaces $U^p(L^2_x)$, where again $1\le p < \infty$. Let $P=\{t_0 < \dots < t_K\}$ be a
partition as above and $\psi_1, \dots , \psi_K \in L^2_x$. Then the step function
$$a=\sum_{k=1}^K\chi_{[t_{k-1},t_k)}\psi_k$$
is called a $U^p$-atom, if $\sum_{k=1}^K\|\psi_k\|^p_{L^2_x}=1$. One says that $u\in U^p(L^2_x)$, if there exist sequences $(\lambda_j)_{j\in\N}\in\ell^1(\N)$
and $(a_j)_{j\in\N}$ of $U^p$-atoms, so that $u=\sum_{j=1}^{\infty}\lambda_j a_j$. These functions constitute a linear space, which endowed with the norm
$$\|u\|_{U^p}:=\inf \Big\{\sum_{j=1}^{\infty}|\lambda_j|: u= \sum_{j=1}^{\infty}\lambda_j a_j\Big\}$$
becomes a B-space. If $1\le p < q < \infty$ the embeddings
$$U^p(L^2_x) \subset U^q(L^2_x) \subset L^{\infty}_I(L^2_x)$$
are continuous. $U^p$-functions are continuous from the right. These two scales of function spaces are tied by continuous embeddings.
Assume once more $1\le p < q < \infty$. Then we have
\begin{equation}\label{UVemb}
 U^p(L^2_x) \subset V^p(L^2_x) \qquad \qquad \mbox{and} \qquad \qquad V_{rc}^p(L^2_x) \subset U^q(L^2_x).
\end{equation}
Comparing with Besov-norms (of $L^2_x$-valued functions) we have the inequalities
$$\|v\|_{\dot{B}^{\frac{1}{p}}_{p,\infty}} \ls \|v\|_{V^p}\qquad \qquad \mbox{and} \qquad \qquad \|u\|_{U^p} \ls \|u\|_{\dot{B}^{\frac{1}{p}}_{p,1}}.$$
Apart from the embeddings above, the $U^p$'s and $V^p$'s are connected by duality. In fact for $1<p<\infty$ and $\frac{1}{p}+\frac{1}{p'}=1$ we can identify
\begin{equation}\label{UVdual}
 (U^p(L^2_x))' \simeq V^{p'}(L^2_x) ,
\end{equation}
where the dual pairing $B: U^p(L^2_x) \times V^{p'}(L^2_x) \to \C$ is given by a generalized Stieltjes integral
$$B(u,v)=\int_I \langle u,dv\rangle =\int_I \langle u, \frac{\partial v}{\partial t}\rangle dt,$$
the latter, if $v$ has a locally integrable weak derivative. We refer to Sections 3.5, 3.8, and 3.10 of \cite{Koch} for details on the duality
between $U^p$ and $V^{p'}$.

\quad

Now Bourgain's construction of the $X^{s,b}$-spaces is repeated. Let $\varphi: \R^n \to \R$ be a phase function and $(U_{\varphi}(t))_{t\in\R}$
the unitary group associated with the linear equation $u_t=i\varphi(-i\nabla)u$. Then one defines
$$V^p_{\varphi}=U_{\varphi}V^{p}(L^2_x)\qquad \qquad \mbox{and} \qquad \qquad U^p_{\varphi}=U_{\varphi}U^{p}(L^2_x)$$
with norms $\|v\|_{V^p_{\varphi}}=\|U_{\varphi}(-\cdot)v\|_{V^p}$ and $\|u\|_{U^p_{\varphi}}=\|U_{\varphi}(-\cdot)u\|_{U^p}$. As for Bourgain's
spaces, an immediate consequence of this definition is the equality $\|U_{\varphi}u_0\|_{V^p_{\varphi}} = \|u_0\|_{L^2_x}$, which we shall
frequently use. The duality \eqref{UVdual} gives us the estimate
$$\|\int_0^tU_{\varphi}(t-s)F(s)ds\|_{V^{p'}_{\varphi}}=\sup_{ \|w\|_{U^p_{\varphi}}\le 1}\left|\int_{I\times\R^n} F(x,t)w(x,t)dxdt \right|$$
for the solution of the inhomogeneous linear equation, cf. Lemma 3.33 in \cite{Koch}. As the $X^{s,b}$-spaces, the spaces $U^p_{\varphi}$
admit a transfer principle. A Strichartz type estimate
$$\|U_{\varphi}u_0\|_{L_t^pL_x^q} \ls \|u_0\|_{L^2_x} \qquad \qquad \mbox{implies} \qquad \qquad \|u\|_{L_t^pL_x^q} \ls \|u\|_{U^p_{\varphi}},$$
if the order of integration is reversed, we have
$$\|U_{\varphi}u_0\|_{L_x^pL_t^q} \ls \|u_0\|_{L^2_x} \qquad \qquad \mbox{implies} \qquad \qquad \|u\|_{L_x^pL_t^q} \ls \|u\|_{U^r_{\varphi}},$$
where $r=\min{(p,q)}$. The $U^p_{\varphi}$- and $U^r_{\varphi}$-norms on the right can be further estimated in $V^2_{\varphi}$, provided
$p>2$ and $r>2$, due to the continuous embedding $V^2_{rc}\subset U^{2+}$. A multilinear version of the transfer principle holds true
as well, see \cite[Proposition 2.19]{HHK}, but we will not make use of it here.

\quad

We now take $I=\R$ and specify $\varphi$ to the phase function
$$\phi: \R \times \R^{n-1} \to \R, \qquad (\xi,\eta)\mapsto \phi(\xi,\eta)=\xi(\xi^2+|\eta|^2)$$
corresponding to the linear part of ZK. For $1\le q < \infty$ we introduce
$$\|u\|_{\dot{X}^s_q}:=\Big( \sum_{N \in 2^{\Z}}N^{sq}\|P_Nu\|^q_{V^2_{\phi}}\Big)^{\frac{1}{q}},$$
which we modify for $q=\infty$ in the usual way, i. e.
$$\|u\|_{\dot{X}^s_{\infty}}:=\sup_{N \in 2^{\Z}}N^{s}\|P_Nu\|_{V^2_{\phi}}.$$
Then we define the $B$-spaces
$$\dot{X}^s_q:=\{u\in C(\R,\dot{B}^s_{2,q}):\|u\|_{\dot{X}^s_q} < \infty\},$$
if $1\le q < \infty$, and 
$$\dot{X}^s_{\infty}:=\{u\in C(\R,\dot{B}^s_{2,\infty}):\|u\|_{\dot{X}^s_{\infty}} < \infty,
\lim_{N\to\infty}N^{s}\|P_Nu\|_{V^2_{\phi}}=\lim_{N\to0}N^{s}\|P_Nu\|_{V^2_{\phi}}=0\}.$$
By the limit conditions the latter is adapted to our data space $\dot{B}^{s,o}_{2,\infty}$. We emphasize that here and below the Littlewood-Paley projections are always applied with respect to \emph{all}
space variables, which we can fix in the form
$$P_N=\F_{xy}^{-1}\chi_{\{|(\xi,\eta)|\sim N\}}\F_{xy}.$$
If in the above the real axis is replaced by a time interval $I=[0,T)$, we write $\dot{X}^s_{q,T}$ (instead of $\dot{X}^s_q$), which
are for $s=s_c$ our solution spaces. Here $T=\infty$ is admitted for the global result. In the proof, and already to make our local
statement precise, we need an auxiliary norm, which depends on $k$ and on the space dimension. It is motivated by the linear estimates
we shall use. Let $I_x$ and $I_y$ denote the Riesz potential operators of order $-1$ with respect to $x\in\R$ and $y\in\R^{n-1}$,
respectively.
\begin{itemize}
 \item For $n=2$ we define $K(I_x,I_y)^{\sigma}=\F_{xy}^{-1}|3\xi^2-\eta^2|^{\sigma}\F_{xy}$ and set
$$|P_N u|_{(k)}:= N^{s_c}\|K(I_x,I_y)^{\frac18}P_N u\|_{L^4_{xyt}}+N^{s_c} \|P_N u\|_{L^4_{xy}L^6_t}+N^{\frac12-\frac{5}{4k}}\|P_N u\|_{L^4_{xy}L^{4k}_t},$$
where $s_c=1-\frac{2}{k}$.
 \item For $n=3$ we have $s_c=\frac32-\frac{2}{k}$ and define
$$|P_N u|_{(k)}:= N^{s_c} \|I_x^{\frac{1}{10}}P_N u\|_{L^{\frac{15}{4}}_{xyt}}+N^{s_c}\|P_N u\|_{L^4_{xyt}}
+N^{\frac34-\frac{3}{2k}}\|P_N u\|_{L^4_{xy}L^{6k}_t}.$$
\end{itemize}
From these we build up the Besov type norms $\| u\|_{(k,q)}=(\sum_{N\in2^{\Z}}|P_N u|^q_{(k)})^{\frac{1}{q}}$, with the usual modification
for $q=\infty$. If the time interval is $[0,T)$ instead of $\R$, we write $|P_N u|_{(k,T)}$ and $\| u\|_{(k,q,T)}$, respectively. The linear
estimates in Sections 4.2 and 5.1 imply via the transfer principle that
$$|P_N u|_{(k)}\ls N^{s_c}\|P_Nu\|_{V^2_{\phi}} \qquad \qquad \mbox{and hence} \qquad \qquad \| u\|_{(k,q)}\ls \| u\|_{\dot{X}^{s_c}_{q}}.$$
Since all H\"older exponents in $|\cdot|_{(k)}$ are finite, this has the consequence, that for all $u \in \dot{X}^{s_c}_{q}$ we have
$\lim_{T\to0}\| u\|_{(k,q,T)}=0$. For $q=\infty$ this is due to our assumption $\lim_{N\to\infty}N^{s}\|P_Nu\|_{V^2_{\phi}}=\lim_{N\to0}N^{s}\|P_Nu\|_{V^2_{\phi}}=0$
in the definition of $\dot{X}^s_{\infty}$. Now our main result takes the following shape.
\begin{theorem}\label{main}
 Let $n\in \{2,3\}$, $k\ge 3$ an integer and $s_c=\frac{n}{2}-\frac{2}{k}$. Assume $u_0 \in \dot{B}^{s_c}_{2,q}$, if $q<\infty$, or $u_0 \in \dot{B}^{s_c,o}_{2,\infty}$. Then
\begin{enumerate}
 \item there exists a $T>0$ and a unique solution $u\in\dot{X}^{s_c}_{q,T}$ of \eqref{gZK} with $u(0)=u_0$. Moreover, there exists a constant
$C=C(k,q)>0$, so that the lifespan of solutions can be chosen uniformly equal to $T$ on the subset
$$D_T:=\{u_0: \|U_{\phi}u_0\|_{(k,q,T)}\le(4C)^{-k}\}$$
of the data space, and the map
$$S_T:D_T \to \dot{X}^{s_c}_{q,T}, \qquad \qquad u_0 \mapsto S_Tu_0:=u$$
(data upon solution) is Lipschitz continuous.
 \item there exists $\e=\e(k,q)>0$ such that, if $\|u_0\|_{\dot{B}^{s_c}_{2,q}}\le \e$, there exists a unique global solution $u\in\dot{X}^{s_c}_{q,\infty}$
of \eqref{gZK} with $u(0)=u_0$. The solution map $S_{\infty}$ is Lipschitz continuous from the ball $B_{\e}:=\{u_0:\|u_0\|_{\dot{B}^{s_c}_{2,q}}\le \e\}$
(contained in the data space) into $\dot{X}^{s_c}_{q,\infty}$.
\end{enumerate}

\end{theorem}

\section{Symmetrization and linear estimates in two space dimensions}

In \cite[Section 2.1]{GH} we observed, that \emph{in two space dimensions} the Zakharov-Kuznetsov equation can be symmetrized
by a linear change of the space variables. For that purpose we fixed $\mu:=4^{-\frac13}$, $\lambda:=\sqrt{3}\mu$ and introduced
$$R_0: \R^2 \to \R^2, \qquad (x,y)\mapsto (x',y'):=(\mu x + \lambda y, \mu x - \lambda y)$$
as well as $Rv:=v\circ R_0$. Let $u=Rv$. Then $u$ is a solution of the Zakharov-Kuznetsov equation
\begin{equation}\label{3.1}
  \partial_t u +\partial_x(\partial^2_x+\partial^2_{y})u=c_0\partial_xu^{k+1}
\end{equation}
with initial condition $u(0)=u_0$, iff $v$ solves
\begin{equation}\label{3.2}
 \partial_t v +(\partial^3_x+\partial^3_{y})v=\mu c_0(\partial_x+\partial_y)v^{k+1}\qquad \mbox{and} \qquad v(0)=R^{-1}u_0.
\end{equation}
For more details, see \cite{GH}\footnote{We also refer to the systematic treatise on linear transformations in connection with dispersive estimates
for third order equations in two space dimensions by Ben-Artzi, Koch, and Saut, in \cite{BKS}.}. The exponent $2$ in
that paper can be replaced without any other changes by $k+1$. For the study of linear estimates we may choose $c_0=0$ in \eqref{3.1}
and \eqref{3.2}. The map $R$ introduced above defines an isomorphism on any of the spaces $H^s(\R^2)$, $\dot{H}^s(\R^2)$, $B^s_{p,q}(\R^2)$,
$\dot{B}^s_{p,q}(\R^2)$ and on all mixed Lebesgue-spaces of the types $L^p_{xy}L^r_t$ and $L^p_tL^r_{xy}$. Thus especially the
well-posedness theory remains unchanged if we pass over from \eqref{3.1} to \eqref{3.2}.

\subsection{Estimates for the linear part of the symmetrized equation, the maximal function estimate} \hfill\\

For the solution of \eqref{3.2} with $c_0=0$ and initial datum $v_0$ we write $U_{\varphi_{sym}}(t)v_0$. With familiar notation we have
$U_{\varphi_{sym}}(t)=e^{-t(\partial^3_x+\partial^3_{y})}$ or, by using the Fourier transform in the space variables,
$$U_{\varphi_{sym}}(t)v_0(x,y)=\int_{\R^2}e^{i(x\xi+y\eta+t\varphi_{sym}(\xi,\eta))}\widehat{v}_0(\xi,\eta)d \xi d \eta,$$
where $\varphi_{sym}(\xi,\eta)=\xi^3+\eta^3$ is the phase function associated with the symmetrized linear problem (i. e. \eqref{3.2}
with $c_0=0$ and initial datum $v_0$). The main advantage of the symmetrization is, that it allows us to obtain the following maximal
function estimate, which is global in time and avoids any technical loss of derivatives.

\begin{prop}\label{max}
 Let $v_0 \in \dot{H}^{\frac12}(\R^2)$ and $f \in L^{\frac43}_{xy}L^1_t(\R^3)$. Then the estimates
\begin{equation}\label{3.3}
\|(I_xI_y)^{-\frac12}\int_0^tU_{\varphi_{sym}}(t-s)f(\cdot, \cdot,s)ds\|_{L^4_{xy}L_t^{\infty}} \ls \|f\|_{L^{\frac43}_{xy}L^1_t}
\end{equation}
and
\begin{equation}\label{3.4}
 \|U_{\varphi_{sym}}v_0\|_{L^4_{xy}L_t^{\infty}} \ls \|(I_xI_y)^{\frac14}v_0\|_{L^2_{xy}}
\end{equation}
hold true.
\end{prop}

\begin{proof}
 We use and follow the arguments in \cite[Section 3]{KPV93}. From Lemma 3.6 in that reference we know the estimate
$$\left| \int_{-\infty}^{\infty}e^{i(x\xi+t\xi^3)}\frac{d\xi}{|\xi|^{\frac12}}\right| \ls |x|^{-\frac12},$$
where the oscillatory integral should be understood as
$$\lim_{\e \to 0, \e > 0}\int_{-\infty}^{\infty}e^{i(x\xi+t\xi^3)-\e\xi^2}\frac{d\xi}{|\xi|^{\frac12}}.$$
This interpretation allows us to use Fubini's theorem for a double integral of this kind, and we obtain
$$\left| \int_{-\infty}^{\infty}\int_{-\infty}^{\infty}e^{i(x\xi+y\eta+t(\xi^3+\eta^3))}\frac{d\xi d\eta}{|\xi\eta|^{\frac12}}\right| \ls |xy|^{-\frac12}.$$
Thus we have
$$\left| \int_{-\infty}^{\infty}(I_xI_y)^{-\frac12}U_{\varphi_{sym}}(t-s)f(x,y,s)ds \right| \ls |xy|^{-\frac12}*\int_{-\infty}^{\infty}|f(x,y,s)|ds.$$
Here the convolution is done with respect to the space variables $x$ and $y$, and the bound on the right is independent of $t$. Now the Hardy-Littlewood-Sobolev
inequality is applied to obtain
\begin{equation}\label{3.5}
 \|\sup_{t \in \R}|\int_{-\infty}^{\infty}(I_xI_y)^{-\frac12}U_{\varphi_{sym}}(t-s)f(\cdot,\cdot,s)ds|\|_{L^4_{xy}} \ls \|f\|_{L^{\frac43}_{xy}L^1_t}.
\end{equation}
 Replacing $f$ by $\chi_{[0,t]}f$ in \eqref{3.5}, we obtain \eqref{3.3}, finally the $TT^*$-argument leads to \eqref{3.4}.
\end{proof}

Next we recall the Strichartz estimates with derivative gain. From \cite[Theorem 3.1]{KPV91} we obtain as a special case
\begin{equation}\label{3.6}
 \|(I_xI_y)^{\frac{1}{2p}}U_{\varphi_{sym}}v_0\|_{L_t^pL^q_{xy}}\ls \|v_0\|_{L^2_{xy}},
\end{equation}
provided $2<p\le \infty$ and $\frac{1}{p}+\frac{1}{q}=\frac{1}{2}$. As usual, the endpoint case $p=2$ is excluded here. Interpolation of
\eqref{3.4} and the $p=q=4$ - case of \eqref{3.6} leads to

\begin{kor}\label{maxStrichartz}
 Let $4 \le r \le \infty$ and $\sigma = \frac14-\frac{3}{2r}$. Then the estimate

\begin{equation}\label{3.7}
 \|U_{\varphi_{sym}}v_0\|_{L^4_{xy}L^r_t} \ls \|(I_xI_y)^{\sigma}v_0\|_{L^2_{xy}}
\end{equation}
holds true. Moreover we have for $6 \le r \le \infty$ and $s=\frac12-\frac{3}{r}$
\begin{equation}\label{3.8}
 \|U_{\varphi_{sym}}v_0\|_{L^4_{xy}L^r_t} \ls \|v_0\|_{\dot{H}^{s}_{xy}}.
\end{equation}
\end{kor}

In order to control the derivative in the nonlinearity we will use the sharp version of the local smoothing effect, which -- by the 
product structure of $U_{\varphi_{sym}}(t)=e^{-t\partial_x^3}e^{-t\partial_y^3}$ -- can be easily deduced from the one of the Airy
equation. For that purpose we fix both space variables $x$ and $y$ and recall from the proof of \cite[Theorem 3.5]{KPV93} the identity
$$ \|I_x e^{-t\partial_x^3}v_0(x,y)\|^2_{L^2_t} = c^2 \|v_0(\cdot,y)\|^2_{L^2_x}.$$
Integration with respect to $y$ gives
$$ \|I_x e^{-t\partial_x^3}v_0(x,\cdot)\|^2_{L^2_{yt}} = c \|v_0\|^2_{L^2_{xy}},$$
which combined with the unitarity of $e^{-t\partial_y^3}$ on $L^2_y$ leads to
$$ \|I_x e^{-t(\partial_x^3+\partial_y^3)}v_0(x,\cdot)\|^2_{L^2_{yt}}=\|I_x e^{-t\partial_x^3}v_0(x,\cdot)\|^2_{L^2_{yt}}= c \|v_0\|^2_{L^2_{xy}}.$$
Thus and by symmetry we have shown the identities
\begin{equation}\label{Katox}
 \|I_xU_{\varphi_{sym}}v_0\|_{L_x^{\infty}L^2_{yt}} = c \|v_0\|_{L^2_{xy}}
\end{equation}
and
\begin{equation}\label{Katoy}
 \|I_yU_{\varphi_{sym}}v_0\|_{L_y^{\infty}L^2_{xt}} = c \|v_0\|_{L^2_{xy}}.
\end{equation}
This smoothing effect is relatively weak in so far, as we get control over $I_x$ or $I_y$, respectively, but not over the full gradient
in both space variables, compare with \eqref{KatoFaminskii} below. This might be seen as the price to pay for the symmetrization.

\subsection{Linear estimates for the original ZK equation in $2$ $D$} \hfill\\

In the subsequent analysis we will apply the linear estimates of the previous section to treat the symmetrized gZK equation \eqref{3.2}.
Nonetheless it may be of interest to compare them with known linear estimates for the original ZK equation, and especially to trace back
the Strichartz- and maximal function estimates by the aid of the transformation $R$ from above. So let $\phi(\xi,\eta)=\xi(\xi^2+\eta^2)$
be the phase function and $U_{\phi}(t)$ the propagator associated with (the linear part of) equation \eqref{3.1}. \\

The local smoothing estimate
\begin{equation}\label{KatoFaminskii}
 \|IU_{\phi}u_0\|_{L_x^{\infty}L^2_{yt}}\ls \|u_0\|_{L^2_{xy}}
\end{equation}
was shown in this case by Faminskii, see \cite[Theorem 2.2]{F95}. Here $I$ denotes the Riesz potential operator of order $-1$ with
respect to $x$ \emph{and} $y$, so there is the in comparison with \eqref{Katox} stronger gain of the full gradient in
this estimate. (Since \eqref{Katox} is sharp and in view of the transformation $R$, this difference seems somewhat surprising -- at
least it was for the author. But it merely reflects the fact, that $R$ is not well-behaved as a mapping on mixed Lebesgue spaces
of type $L^p_xL^q_y$, if $p \ne q$.) \\

To convert the Strichartz- and maximal function estimates for $v$ into estimates for $u=Rv=v\circ R_0$ in terms of $u_0=v_0 \circ R_0$,
we apply the Fourier transform to the last identity and obtain $\widehat{u}_0=|\det R_0|^{-1}\widehat{v}_0 \circ (R_0^{\top})^{-1}$,
hence 
$$\widehat{v}_0(\xi',\eta')=|\det R_0|\widehat{u}_0 \circ R_0^{\top}(\xi',\eta').$$
We set $(\xi,\eta)=R_0^{\top}(\xi',\eta')$ and multiply both sides by $|\xi' \eta '|^{\sigma}= c_{\sigma}|3\xi^2 -\eta^2|^{\sigma}$. Then
$$|\xi' \eta '|^{\sigma}\widehat{v}_0(\xi',\eta')=c |3\xi^2 -\eta^2|^{\sigma}\widehat{u}_0 (\xi,\eta)$$
which can be squared and integrated with respect to $d\xi'd\eta' = |\det R_0|^{-1}d\xi d\eta$. With the Fourier multiplier
$$ K(I_x,I_y)^{\sigma}:=\F^{-1}_{xy}|3\xi^2 -\eta^2|^{\sigma}\F_{xy} $$
we then have $\|(I_xI_y)^{\sigma}v_0\|_{L^2_{xy}}=c\|K(I_x,I_y)^{\sigma}u_0\|_{L^2_{xy}}$. This gives the following Strichartz type
inequality for $u$.
\begin{equation}\label{ZKStrichartz}
  \|K(I_x,I_y)^{\frac{1}{2p}}U_{\phi}u_0\|_{L_t^pL^q_{xy}}\ls \|u_0\|_{L^2_{xy}},
\end{equation}
provided $2<p\le \infty$ and $\frac{1}{p}+\frac{1}{q}=\frac12$. If $p=q=4$, we can recognize this as the special case of (the dual
estimate to) Theorem 1.1 in \cite{CKZ13} by Carbery, Kenig, and Ziesler, which has been applied by Molinet and Pilod in their work
\cite{MP15} on the ZK equation, cf. Proposition 3.5 in that paper. In fact, our simple considerations here give a wider range of
validity with a stronger gain of derivatives, if $p \to 2$. Again the endpoint $p=2$ is excluded. \\

Similarly, we have the maximal function estimate
\begin{equation}\label{ZKmax}
 \|U_{\phi}u_0\|_{L^4_{xy}L_t^{\infty}} \ls \|K(I_x,I_y)^{\frac14}u_0\|_{L^2_{xy}} \le \|u_0\|_{\dot{H}^{\frac12}_{xy}}.
\end{equation}
A Sobolev embedding in the $y$-Variable gives
\begin{equation}\label{ZKmaxx}
\|U_{\phi}u_0\|_{L^4_{x}L_{yt}^{\infty}} \ls \|K(I_x,I_y)^{\frac14}J_y^{\frac14 +}u_0\|_{L^2_{xy}},
\end{equation}
which is comparable with Proposition 1.5  in \cite{LP09} by Linares and Pastor. see also Corollary 2.7 of \cite{LP11}. The advantage here is, that \eqref{ZKmaxx}
holds globally in time.

\section{Discussion of the  $2$ $D$ - case}

Throughout this section we have $n=2$ and hence $s_c=1-\frac{2}{k}$. $U^p_{\varphi_{sym}}$ and $V^p_{\varphi_{sym}}$ will denote the $U^p$ - and
$V^p$ - spaces associated with the phase function
$$\varphi_{sym}: \R^2 \to \R, \qquad (\xi,\eta)\mapsto \varphi_{sym}(\xi,\eta)=\xi^3+\eta^3.$$
The solution space $\dot{X}^{s_c}_q$ is that one with norm built on $V^2_{\varphi_{sym}}$.

\subsection{The central multilinear estimate} \hfill\\

We will start with a multilinear estimate on dyadic pieces of functions $v_1, \dots , v_{k+1} \in V^2_{\varphi_{sym}}$. We recall the quantity
$$|P_N v|_{(k)}:= N^{s_c}\|(I_xI_y)^{\frac18}P_N v\|_{L^4_{xyt}}+N^{s_c} \|P_N v\|_{L^4_{xy}L^6_t}+N^{\frac12-\frac{5}{4k}}\|P_N v\|_{L^4_{xy}L^{4k}_t}.$$
We remark that by the transfer principle and the linear estimates \eqref{3.7} and \eqref{3.8} all three contributions are bounded by
the $V^2_{\varphi_{sym}}$ - norm, more precisely we have
$$|P_N v|_{(k)} \ls N^{s_c}\|P_N v\|_{V^2_{\varphi_{sym}}}.$$
If the mixed norms are replaced by $L^4_{xy}L^r_T$, which means that the integration is restricted to the time interval $(0,T)$, we
write $|P_N v|_{(k,T)}$ for the corresponding composed quantity. Here we may have $T=\infty$. Observe that $\lim_{T \to 0}|P_N v|_{(k,T)}=0$,
whenever $v\in V^2_{\varphi_{sym}}$.

\begin{lemma}\label{multpieces2D}
 Let $v_1, \dots , v_{k+1} \in V^2_{\varphi_{sym}}$, $w \in U^2_{\varphi_{sym}}$ with $\|w\|_{U^2_{\varphi_{sym}}}\le 1$,
$N, N_1, \dots , N_{k+1}$ dyadic numbers with $N_1 \le N_2 \le \dots \le N_{k+1}$ and $N \ls N_{k+1}$. Then there exists 
$\e>0$ such that
$$ N^{s_c} \left| \int_{\R^3}P_{N_1}v_1 \cdot ... \cdot P_{N_{k+1}}v_{k+1} \cdot \partial_x P_Nw dxdydt \right| 
\ls  N^{\e}N_1^{\e}N_{k+1}^{-2\e} \prod_{j=1}^{k+1}|P_{N_j} v_j|_{(k)}.$$
The same holds true, if $\partial_x$ is replaced by $\partial_y$.
\end{lemma}

\begin{proof}
 We consider three cases, depending on the relative sizes of the spatial frequencies $(\xi_{k+1},\eta_{k+1})$ and $(\xi,\eta)$ of
$v_{k+1}$ and $w$, respectively. Observe that by our assumptions $|(\xi,\eta)| \ls |(\xi_{k+1},\eta_{k+1})|$. In the sequel, let $\e$ be
a positive number, which has to be chosen sufficiently small in dependence of $k$.\\

\begin{itemize}
 \item[Case 1:] $|\eta_{k+1}| \ls |\xi_{k+1}|$. Here we may replace the factor $P_{N_{k+1}}v_{k+1}$ by $N_{k+1}^{-\e}I_x^{\e}P_{N_{k+1}}v_{k+1}$. \\

 \item[Case 2:] $|\xi_{k+1}| \ll |\eta_{k+1}|$ and $|\xi|\ls|\eta|$. Here we may replace $(P_{N_{k+1}}v_{k+1})(\partial_x P_Nw)$
by $(N_{k+1}^{-\e}I_y^{\e}P_{N_{k+1}}v_{k+1})(\partial_y P_Nw)$ and argue as in Case 1 with the roles of $x$ and $y$ (respectively of the $\xi$'s and $\eta$'s) interchanged.\\

 \item[Case 3:] $|\xi_{k+1}| \ll |\eta_{k+1}|$ and $|\eta|\ls|\xi|$. Since in this case $|\xi| \ls N_{k+1} \sim |\eta_{k+1}|$, we have
$|\eta| \ll |\eta_{k+1}|$. By the convolution constraint $|\sum_{j=1}^{k+1}\eta_j|=|\eta|$, there exists at least one $j \in \{1, \dots , k\}$
with $|\eta_j|\sim|\eta_{k+1}|$. This implies $N_K \sim N_{k+1}$ and especially $N \ls N_k$.\\
\end{itemize}

\emph{Treatment of Case 1:} The contribution from this case is bounded by
\begin{eqnarray*}
 & \displaystyle N^{\e}N_{k+1}^{s_c-2\e}\left| \int_{\R^3}P_{N_1}v_1 \cdot ... \cdot P_{N_k}v_k \cdot (I_x^{\e} P_{N_{k+1}}v_{k+1}) \cdot \partial_x P_Nw dxdydt \right| \\
\displaystyle \ls &  N^{\e}N_{k+1}^{s_c-2\e} \|P_{N_1}v_1 \cdot ... \cdot P_{N_k}v_k \cdot (I_x^{\e} P_{N_{k+1}}v_{k+1})\|_{L_x^1L^2_{yt}}\|\partial_x P_Nw\|_{L_x^{\infty}L^2_{yt}},
\end{eqnarray*}
where by the local smoothing effect \eqref{Katox} and the transfer principle the last factor is bounded by $\|P_Nw\|_{U^2_{\varphi_{sym}}} \le 1$.
For $\e >0$ sufficiently small we choose
$$ \frac{1}{p_1}=\frac{3}{4k}, \qquad \qquad \qquad \frac{1}{q_1}=\frac{1}{4k}+\e, \qquad \qquad \qquad \frac{1}{r_1}=\frac{1}{3k}-\frac{2\e}{3}, $$
and, for $j \in \{2, \dots , k\}$,
$$ \frac{1}{p_j}=\frac{3}{4k}, \qquad \qquad \qquad \frac{1}{q_j}=\frac{1}{4k}, \qquad \qquad \qquad \frac{1}{r_j}=\frac{1}{3k}, $$
as well as
$$ \frac{1}{p_{k+1}}=\frac{1}{4}, \qquad \qquad \qquad \frac{1}{q_{k+1}}=\frac{1}{4}-\e, \qquad \qquad \qquad \frac{1}{r_{k+1}}=\frac{1}{6}+\frac{2\e}{3}. $$
Since $k \ge 3$, we have $p_j \ge 4$ and $q_j \ge 4$ for all $j \in \{1, \dots ,k+1\}$. Moreover $\displaystyle \sum_{j=1}^{k+1}\frac{1}{p_j}=1$ and $\displaystyle \sum_{j=1}^{k+1}\frac{1}{q_j}=\sum_{j=1}^{k+1}\frac{1}{r_j}=\frac12$,
so that H\"older's inequality gives the upper bound
\begin{equation}\label{ub1}
 N^{\e}N_{k+1}^{-2\e} \Big(\prod_{j=1}^k \|P_{N_j}v_j\|_{L_x^{p_j}L_y^{q_j}L_t^{r_j}}\Big)N_{k+1}^{s_c}\|I_x^{\e}P_{N_{k+1}}v_{k+1}\|_{L_x^{p_{k+1}}L_y^{q_{k+1}}L_t^{r_{k+1}}}.
\end{equation}
For the $v_1$ - factor we use Sobolev embeddings in the space variables to obtain 
$$\|P_{N_1}v_1\|_{L_x^{p_1}L_y^{q_1}L_t^{r_1}} \ls N_1^{\frac12 - \frac{1}{k} - \e}\|P_{N_1}v_1\|_{L^4_{xy}L_t^{r_1}}=N_1^{\e} N_1^{\frac12 - \frac{1}{k} -2 \e}\|P_{N_1}v_1\|_{L^4_{xy}L_t^{r_1}}.$$
We choose $\theta$ so that $\frac{1}{r_1}=\frac{1-\theta}{6}+\frac{\theta}{4k}$. Then Lyapunov's inequality gives the bound
$$N_1^{\frac12 - \frac{1}{k} -2 \e}\|P_{N_1}v_1\|_{L^4_{xy}L_t^{r_1}} \le \Big(N_1^{s_c}\|P_{N_1}v_1\|_{L^4_{xy}L_t^{6}}\Big)^{1-\theta}\Big(N_1^{\frac12-\frac{5}{4k}}\|P_{N_1}v_1\|_{L^4_{xy}L_t^{4k}}\Big)^{\theta},$$
which in turn is dominated by
$$N_1^{s_c}\|P_{N_1}v_1\|_{L^4_{xy}L_t^{6}}+N_1^{\frac12-\frac{5}{4k}}\|P_{N_1}v_1\|_{L^4_{xy}L_t^{4k}}\le |P_{N_1}v_1|_{(k)}.$$
Collecting terms we obtain 
$$\|P_{N_1}v_1\|_{L_x^{p_1}L_y^{q_1}L_t^{r_1}} \ls N_1^{\e}|P_{N_1}v_1|_{(k)},$$
and, taking $\e=0$ in this calculation for $v_1$, we as well have for $j \in \{2, \dots , k\}$ that
$$\|P_{N_j}v_j\|_{L_x^{p_j}L_y^{q_j}L_t^{r_j}} \ls |P_{N_j}v_j|_{(k)}.$$
For the last factor we use a Sobolev embedding with respect to the $y$ - variable and a convexity inequality (like Lyapunov's
inequality above) to obtain
\begin{eqnarray*}\|I_x^{\e}P_{N_{k+1}}v_{k+1}\|_{L_x^{p_{k+1}}L_y^{q_{k+1}}L_t^{r_{k+1}}}\ls \|(I_xI_y)^{\e}P_{N_{k+1}}v_{k+1}\|_{L_{xy}^{4}L_t^{r_{k+1}}}\\
\ls \|(I_xI_y)^{\frac18}P_{N_{k+1}}v_{k+1}\|_{L^4_{xyt}}+ \|P_{N_{k+1}}v_{k+1}\|_{L_{xy}^{4}L_t^{6}} \qquad \quad ,
\end{eqnarray*}
so that
$$N_{k+1}^{s_c}\|I_x^{\e}P_{N_{k+1}}v_{k+1}\|_{L_x^{p_{k+1}}L_y^{q_{k+1}}L_t^{r_{k+1}}}\ls |P_{N_{k+1}}v_{k+1}|_{(k)}.$$
Summarizing the estimates for the single factors, we see that \eqref{ub1} is in fact bounded by
$$N^{\e}N_1^{\e}N_{k+1}^{-2\e} \prod_{j=1}^{k+1}|P_{N_j} v_j|_{(k)},$$
as desired.\\

\emph{Estimation for Case 3:} Since $N_k \sim N_{k+1}$ here, the contribution is bounded by
$$N^{\e}N_k^{\e}N_{k+1}^{s_c-2\e} \|P_{N_1}v_1 \cdot ... \cdot P_{N_{k+1}}v_{k+1}\|_{L_x^1L^2_{yt}}.$$
For $j \in \{1, \dots , k-1\}$ we choose H\"older exponents $p_j$, $q_j$, and $r_j$ precisely as in Case 1. Moreover we set
$$ \frac{1}{p_k}=\frac{3}{4k}, \qquad \qquad \qquad \frac{1}{q_k}=\frac{1}{4k}-\e, \qquad \qquad \qquad \frac{1}{r_k}=\frac{1}{3k}+\frac{2\e}{3}, $$
as well as
$$ \frac{1}{p_{k+1}}=\frac{1}{4}, \qquad \qquad \qquad \frac{1}{q_{k+1}}=\frac{1}{4}, \qquad \qquad \qquad \frac{1}{r_{k+1}}=\frac{1}{6}. $$
Then again we have $p_j \ge 4$ and $q_j \ge 4$ for all $j \in \{1, \dots ,k+1\}$ and $\displaystyle \sum_{j=1}^{k+1}\frac{1}{p_j}=1$ as well as
$\displaystyle \sum_{j=1}^{k+1}\frac{1}{q_j}=\sum_{j=1}^{k+1}\frac{1}{r_j}=\frac12$. H\"older's inequality gives the only slightly different
upper bound
\begin{equation}\label{ub2}
 N^{\e}N_k^{\e}N_{k+1}^{-2\e} \Big(\prod_{j=1}^k \|P_{N_j}v_j\|_{L_x^{p_j}L_y^{q_j}L_t^{r_j}}\Big)N_{k+1}^{s_c}\|P_{N_{k+1}}v_{k+1}\|_{L_{xy}^{4}L_t^{6}}.
\end{equation}
From the estimates concerning Case 1 we already know that
\begin{equation}\label{number1}
\prod_{j=1}^{k-1} \|P_{N_j}v_j\|_{L_x^{p_j}L_y^{q_j}L_t^{r_j}} \ls N_1^{\e}\prod_{j=1}^{k-1} |P_{N_j}v_j|_{(k)}.
\end{equation}
Moreover it is clear by our choices, that
$$N_{k+1}^{s_c}\|P_{N_{k+1}}v_{k+1}\|_{L_{xy}^{4}L_t^{6}} \le |P_{N_{k+1}}v_{k+1}|_{(k)},$$
and it remains to estimate the factor for $j=k$. Sobolev embeddings in $x$ and $y$ give
$$\|P_{N_k}v_k\|_{L_x^{p_k}L_y^{q_k}L_t^{r_k}} \ls N_k^{\frac12 - \frac{1}{k} + \e}\|P_{N_k}v_k\|_{L^4_{xy}L_t^{r_k}}
=N_k^{-\e} N_k^{\frac12 - \frac{1}{k} +2 \e}\|P_{N_k}v_k\|_{L^4_{xy}L_t^{r_k}}.$$
Applying Lyapunov's inequality again we obtain (replace $\e$ by $-\e$ in the corresponding argument for $v_1$ in Case 1)
$$N_k^{\frac12 - \frac{1}{k} +2 \e}\|P_{N_k}v_k\|_{L^4_{xy}L_t^{r_k}} \ls N_k^{s_c} \|P_{N_k} v_k\|_{L^4_{xy}L^6_t}+N_k^{\frac12-\frac{5}{4k}}\|P_{N_k} v_k\|_{L^4_{xy}L^{4k}_t}\ls |P_{N_k} v_k|_{(k)},$$
so that
\begin{equation}\label{number2}
 \|P_{N_k}v_k\|_{L_x^{p_k}L_y^{q_k}L_t^{r_k}} \ls N_k^{-\e}|P_{N_k} v_k|_{(k)}.
\end{equation}
The comparison of \eqref{number1} and \eqref{number2} shows, that we have successfully exchanged the large factor $N_k^{\e}$ by the smaller $N_1^{\e}$.
Alltogether
$$\eqref{ub2} \ls N^{\e}N_1^{\e}N_{k+1}^{-2\e} \prod_{j=1}^{k+1}|P_{N_j} v_j|_{(k)}.$$
This completes the estimation in Case 3. \\

The statement about $\partial_y$ instead of $\partial_x$ is obvious by symmetry.
\end{proof}

The next step is to sum up these estimates on dyadic pieces, which will necessarily involve the auxiliary norms
$$\|v\|_{(k,q)}= \Big(\sum_{N \in 2^{\Z}} |P_N v|_{(k)}\Big)^{\frac{1}{q}}$$
with the usual modification for $q=\infty$. (We write $\|v\|_{(k,q,T)}$, if these norms are assembled from $(|P_N v|_{(k,T)})_{N \in 2^{\Z}}$.)
They remain finite for $v \in \dot{X}^{s_c}_{q}$ (or $v \in \dot{X}^{s_c}_{q,T}$, respectively). For $v_1, \dots , v_{k+1} \in \dot{X}^{s_c}_{q}$
we introduce
$$F(v_1, \dots , v_{k+1})(t):=\int_0^tU_{\varphi_{sym}}(t-s)(\partial_x)(v_1\cdot ... \cdot v_{k+1})(s)ds,$$
the dependence on $x$ and $y$ was suppressed here.

\begin{lemma}\label{mult2D}
 For $v_1, \dots , v_{k+1} \in \dot{X}^{s_c}_{q}$ we have $F(v_1, \dots , v_{k+1})\in \dot{X}^{s_c}_{q}$ and the estimate
$$\|F(v_1, \dots , v_{k+1})\|_{\dot{X}^{s_c}_{q}} \ls \prod_{j=1}^{k+1}\|v_j\|_{(k,q)}$$
holds true.
\end{lemma}

\emph{Remark:} The statement is still correct, if $\dot{X}^{s_c}_{q}$ and $\|\cdot\|_{(k,q)}$ are replaced by $\dot{X}^{s_c}_{q,T}$ and
$\|\cdot\|_{(k,q,T)}$, respectively, and if $\partial_x$ is changed into $\partial_y$.

\begin{proof}
 By the duality between $V^2_{\varphi}$ and $U^2_{\varphi}$ (cf. \eqref{UVdual} and the subsequent remark) we have
$$\|F(v_1, \dots , v_{k+1})\|_{V^2_{\varphi_{sym}}}= \sup_{\|w\|_{U^2_{\varphi_{sym}}}\le1}\left|\int_{\R^3}v_1\cdot ... \cdot v_{k+1}\partial_xwdxdydt\right| .$$
Thus Lemma \ref{multpieces2D} tells us that for $N_1 \le \dots \le N_{k+1}$
$$N^{s_c}\|P_N F(P_{N_1}v_1, \dots ,P_{N_{k+1}} v_{k+1})\|_{V^2_{\varphi_{sym}}}\ls N^{\e}N_1^{\e}N_{k+1}^{-2\e}\prod_{j=1}^{k+1}|P_{N_j}v_j|_{(k)}.$$
(Here, by the convolution constraint $|(\xi,\eta)|=|\sum_{j=1}^{k+1}(\xi_j,\eta_j)|$ we have only contributions for $N \ls N_{k+1}$.)
We fix $N$ and $N_{k+1} \gs N$ and sum up the geometric series in $N_1 \le \dots \le N_{k}$. This gives
\begin{eqnarray}\label{sumk}
& \displaystyle \sum_{N_1 \le \dots \le N_{k}\atop N_k \le N_{k+1}}N^{s_c}\|P_N F(P_{N_1}v_1, \dots ,P_{N_{k+1}} v_{k+1})\|_{V^2_{\varphi_{sym}}} \nonumber \\ 
 \ls &  N^{\e}N_{k+1}^{-\e}\Big(\prod_{j=1}^{k}\|v_j\|_{(k,\infty)}\Big)|P_{N_{k+1}}v_{k+1}|_{(k)}.
\end{eqnarray}
Now we distinguish between $q=\infty$ and $q<\infty$.

\quad

Case 1: $q=\infty$. Here we simply sum up one last geometric series in $N_{k+1} \gs N$, which leads to
\begin{equation}\label{sumk+1}
 \sum_{N_1 \le \dots \le N_{k+1}\atop N \ls N_{k+1}}N^{s_c}\|P_N F(P_{N_1}v_1, \dots ,P_{N_{k+1}} v_{k+1})\|_{V^2_{\varphi_{sym}}} 
\ls  \prod_{j=1}^{k+1}\|v_j\|_{(k,\infty)}.
\end{equation}
Since this works for all orders of $N_1, \dots, N_{k+1}$, we have for $N$ fixed by the triangle inequality
$$N^{s_c}\|P_N F(v_1, \dots , v_{k+1})\|_{V^2_{\varphi_{sym}}}\ls \prod_{j=1}^{k+1}\|v_j\|_{(k,\infty)}.$$
Taking the supremum over all $N \in 2^{\Z}$ we have achieved the claimed inequality in the case $q=\infty$.

\quad

Going back to \eqref{sumk+1} we see that, since the sums over $N_1, \dots, N_{k+1}$ exist,
$$\lim_{N\to\infty} \sum_{N_1 \le \dots \le N_{k+1}\atop N \ls N_{k+1}}N^{s_c}\|P_N F(P_{N_1}v_1, \dots ,P_{N_{k+1}} v_{k+1})\|_{V^2_{\varphi_{sym}}}=0,$$
which implies $\lim_{N\to\infty}N^{s_c}\|P_N F(v_1, \dots , v_{k+1})\|_{V^2_{\varphi_{sym}}}=0$. To see that the limit for $N\to0$ vanishes, too,
let $\delta > 0$ be given. Then there exists $N_{\delta}\in 2^{\Z}$ such that $|P_{N_{k+1}}v_{k+1}|_{(k)}\le \delta$ for all $N_{k+1}\le N_{\delta}$.
Thus (cf. the right hand side of \eqref{sumk})
$$\sum_{N_{k+1}\gs N}\frac{N^{\e}}{N_{k+1}^{\e}}\Big(\prod_{j=1}^{k}\|v_j\|_{(k,\infty)}\Big)|P_{N_{k+1}}v_{k+1}|_{(k)}
\ls \delta \prod_{j=1}^{k}\|v_j\|_{(k,\infty)}+ \frac{N^{\e}}{N_{\delta}^{\e}}\prod_{j=1}^{k+1}\|v_j\|_{(k,\infty)}.$$
Now $N\to0$, then $\delta\to0$.

\quad

To close the discussion in the case $q=\infty$, we have to show that $F(v_1, \dots , v_{k+1})\in C(\R,\dot{B}^{s_c,o}_{2,\infty})$. For
that purpose we fix $t_0\in \R$ and denote the characteristic function of the $t$ - intervall between $t_0$ and $t_0+h$ by $\chi_h$
($h$ may be negative). Then, by the continuous embedding $\dot{X}^{s_c}_{\infty}\subset C(\R,\dot{B}^{s_c,o}_{2,\infty})$ and the estimate
already shown we have
\begin{eqnarray*}
\displaystyle & \|F(v_1, \dots , v_{k+1})(t_0+h)-F(v_1, \dots , v_{k+1})(t_0)\|_{\dot{B}^{s_c}_{2,\infty}} \\
\displaystyle \le & \sup_{t\in\R}\|F(\chi_hv_1, \dots ,\chi_h v_{k+1})(t)\|_{\dot{B}^{s_c}_{2,\infty}} \\
\displaystyle \le & \|F(\chi_hv_1, \dots ,\chi_h v_{k+1})\|_{\dot{X}^{s_c}_{\infty}} \ls \prod_{j=1}^{k+1}\|\chi_hv_j\|_{(k,\infty)},
\end{eqnarray*}
which tends to zero with $h\to 0$.

\quad

Case 2: $q<\infty$. We sum up the right hand side of \eqref{sumk} in $N_{k+1}$ using H\"older's inequality. This gives
\begin{eqnarray*}
 &  \displaystyle\sum_{N_{k+1}\gs N}\frac{N^{\e}}{N_{k+1}^{\e}}\Big(\prod_{j=1}^{k}\|v_j\|_{(k,\infty)}\Big)|P_{N_{k+1}}v_{k+1}|_{(k)} \\
\displaystyle \ls &  \displaystyle \prod_{j=1}^{k}\|v_j\|_{(k,\infty)} \Big(\sum_{N_{k+1}\gs N}N^{\frac{q\e}{2}}N_{k+1}^{-\frac{q\e}{2}}|P_{N_{k+1}}v_{k+1}|^q_{(k)} \Big)^{\frac{1}{q}}.
\end{eqnarray*}
Now we can take the $\ell^q_N(2^{\Z})$ - norm of this and sum up first in $N \ls N_{k+1}$ and then in $N_{k+1}$ to obtain
\begin{eqnarray*}
  &  \displaystyle \Big(\sum_{N\in2^{\Z}} N^{s_cq}\|\sum_{N_1\le \dots \le N_{k+1} \atop N \ls N_{k+1}}P_NF(P_{N_1}v_1, \dots , P_{N_{k+1}}v_{k+1})\|^q_{V^2_{\varphi_{sym}}}\Big)^{\frac{1}{q}} \\
\displaystyle \ls & \displaystyle \Big(\prod_{j=1}^{k}\|v_j\|_{(k,\infty)}\Big)\|v_{k+1}\|_{(k,q)} \qquad  \ls \qquad \prod_{j=1}^{k+1}\|v_j\|_{(k,q)},
\end{eqnarray*}
where in the last step the continuous embedding $\ell^q \subset \ell^{\infty}$ was used. The same bound holds for all orders of
$N_1, \dots , N_{k+1}$, hence we get the claimed inequality. The continuity follows by the same arguments as in Case 1.

\end{proof}

\subsection{Well-posedness for the symmetrized equation} \hfill \\

Here we prove the local and global well-posedness of the Cauchy problem for \eqref{3.2} with initial data $v_0$ in Besov spaces of
critical regularity. By the discussion about symmetrization at the beginning of Section 3.1 this implies the two dimensional part of
Theorem \ref{main}.

\begin{theorem}\label{mainsym}
 Let $v_0 \in \dot{B}^{s_c}_{2,q}$, if $q<\infty$, or $v_0 \in \dot{B}^{s_c,o}_{2,\infty}$. Then
\begin{enumerate}
 \item there exists a $T>0$ and a unique solution $v\in\dot{X}^{s_c}_{q,T}$ of \eqref{3.2} with $v(0)=v_0$. Moreover, there exists a constant
$C=C(k,q)>0$, so that the lifespan of solutions can be chosen uniformly equal to $T$ on the subset
$$D_T:=\{v_0: \|U_{\varphi_{sym}}v_0\|_{(k,q,T)}\le(4C)^{-k}\}$$
of the data space, and the map
$$S_T:D_T \to \dot{X}^{s_c}_{q,T}, \qquad \qquad v_0 \mapsto S_Tv_0:=v$$
(data upon solution) is Lipschitz continuous.
 \item there exists $\e=\e(k,q)>0$ such that, if $\|v_0\|_{\dot{B}^{s_c}_{2,q}}\le \e$, there exists a unique global solution $v\in\dot{X}^{s_c}_{q,\infty}$
of \eqref{3.2} with $v(0)=v_0$. The solution map $S_{\infty}$ is Lipschitz continuous from the ball $B_{\e}:=\{v_0:\|v_0\|_{\dot{B}^{s_c}_{2,q}}\le \e\}$
(contained in the data space) into $\dot{X}^{s_c}_{q,\infty}$.
\end{enumerate}

\end{theorem}
 
\begin{proof}
 For given $v_0$ we search a solution $v=\psi+w$, where $\psi=U_{\varphi_{sym}}v_0$ is a solution of the linear equation with
$\psi(0)=v_0$ and $w$ solves the integral equation $w=\Lambda_{\psi}w$ defined by
$$\Lambda_{\psi}w(t)=\int_0^tU_{\varphi_{sym}}(t-s)(\partial_x+\partial_y)(w+\psi)^{k+1}(s)ds.$$
For $j\in\{1,2\}$ let $v_0^{(j)} \in \dot{B}^{s_c}_{2,q}$ (respectively $v_0^{(j)} \in \dot{B}^{s_c,o}_{2,\infty}$) and
$\psi_j=U_{\varphi_{sym}}v_0^{(j)}$ with $\|\psi_j\|_{(k,q,T)}\le R_0$ as well as $w_j \in \dot{X}^{s_c}_{q,T}$ with
$\|w_j\|_{\dot{X}^{s_c}_{q,T}}\le R$. (The relation between $T$, $R_0$ and $R$ will be specified within the next few lines.) Then by
Lemma \ref{mult2D} and some elementary estimates we obtain
\begin{equation}\label{p1}
 \|\Lambda_{\psi_1}w_1-\Lambda_{\psi_2}w_2\|_{\dot{X}^{s_c}_{q,T}}\le C(R^k+R_0^k)(\|w_1-w_2\|_{\dot{X}^{s_c}_{q,T}}+\|\psi_1-\psi_2\|_{(k,q,T)})
\end{equation}
with a constant $C$, which may only depend on $k$ and $q$. Especially for $w_2=\psi_2=0$ we see that
\begin{equation}\label{p2}
 \|\Lambda_{\psi_1}w_1\|_{\dot{X}^{s_c}_{q,T}}\le C(R^k+R_0^k)(R+R_0),
\end{equation}
if we take $\psi_1=\psi_2$ in \eqref{p1}, we get
\begin{equation}\label{p3}
 \|\Lambda_{\psi_1}w_1-\Lambda_{\psi_1}w_2\|_{\dot{X}^{s_c}_{q,T}}\le C(R^k+R_0^k)\|w_1-w_2\|_{\dot{X}^{s_c}_{q,T}}.
\end{equation}
Now we fix $R=R_0$ in that way, that $CR^k=CR_0^k=\frac14$. Since for any $v_0\in \dot{B}^{s_c}_{2,q}$ (respectively $v_0\in \dot{B}^{s_c,o}_{2,\infty}$)
we have $\lim_{T\to0}\|U_{\phi_{sym}}v_0\|_{(k,q,T)}$, we can reach $\|\psi_j\|_{(k,q,T)}\le R_0$ by choosing $T$ small enough. With this
choice we have
$$\|\Lambda_{\psi_1}w_1\|_{\dot{X}^{s_c}_{q,T}}\le R \qquad \mbox{and} \qquad 
\|\Lambda_{\psi_1}w_1-\Lambda_{\psi_1}w_2\|_{\dot{X}^{s_c}_{q,T}}\le \frac12 \|w_1-w_2\|_{\dot{X}^{s_c}_{q,T}}.$$
Moreover, we know from Lemma \ref{mult2D} that - for $w \in\dot{X}^{s_c}_{q,T}$ - $\Lambda_{\psi_1}w \in\dot{X}^{s_c}_{q,T}$, especially
it is a continuous function with values in the data space. Thus for fixed $\psi_1$ the mapping $\Lambda_{\psi_1}$ is a contraction of
the closed ball of radius $R$ in $\dot{X}^{s_c}_{q,T}$ into itself. The contraction mapping principle provides a solution of $\Lambda_{\psi_1}w=w$,
which is unique in this ball. Since for any $w \in \dot{X}^{s_c}_{q,T}$ we have $\lim_{T\to0}\|w\|_{(k,q,T)}=0$, we can use
a standard argument, to extend the uniqueness property to the whole $\dot{X}^{s_c}_{q,T}$. The statement about the lifespan merely
reflects our choices. These also give, if inserted into \eqref{p1} the inequality
$$\|\Lambda_{\psi_1}w_1-\Lambda_{\psi_2}w_2\|_{\dot{X}^{s_c}_{q,T}}\le \frac12\|w_1-w_2\|_{\dot{X}^{s_c}_{q,T}}+\frac12\|\psi_1-\psi_2\|_{(k,q,T)},$$
which for solutions $w_1=\Lambda_{\psi_1}w_1$ and $w_2=\Lambda_{\psi_2}w_2$ implies the Lipschitz bound
$$\|w_1-w_2\|_{\dot{X}^{s_c}_{q,T}} \le \|\psi_1-\psi_2\|_{(k,q,T)} \ls \|v_0^{(1)}-v_0^{(2)}\|_{\dot{B}^{s_c}_{2,q}}.$$
Clearly, if $v_j=\psi_j + w_j$, we have the same (up to a factor) upper bound for $\|v_1-v_2\|_{\dot{X}^{s_c}_{q,T}}$. Now the local
part of the Theorem is shown. The global part is similar: One uses $\|\psi_j\|_{(k,q,\infty)}\ls\|v_0^{(j)}\|_{\dot{B}^{s_c}_{2,q}}$
and replaces $R_0$ by $\e$ in the inequalities. We omit further details.
\end{proof}

\section{Modifications in the $3$ $D$ - case}

\subsection{Linear estimates in $3$ $D$} \hfill \\

The linear part of the ZK equation in $3$ $D$ is
\begin{equation}\label{ZKlin3d}
 u_t + \partial_x\Delta u = 0,
\end{equation}
where the Laplacian can be written as $\Delta = \frac{\partial^2}{\partial x^2}+\frac{\partial^2}{\partial y_1^2}+\frac{\partial^2}{\partial y_2^2}$
in order to emphasize the symmetry in the second and third space variable. The phase function corresponding to \eqref{ZKlin3d} is
$$\phi(\xi,\eta)= \xi(\xi^2+|\eta|^2) \qquad \mbox{with} \qquad (\xi,\eta)=(\xi,\eta_1,\eta_2)\in \R^3.$$
Let $(U_{\phi}(t))_{t\in\R}$ denote the associated unitary group, so that solutions $u$ of \eqref{ZKlin3d} with initial datum $u_0$
become $u(t,x,y)=U_{\phi}(t)u_0(x,y)$. Then we can rely on various known linear estimates for such solutions. In order to control the
derivative in the multilinear estimates we may use the local smoothing effect of Kato type, i.e.
\begin{equation}\label{Kato3d}
\|IU_{\phi}u_0\|_{L_x^{\infty}L^2_{yt}}\ls \|u_0\|_{L^2_{xy}}.
\end{equation}
Here $I$ denotes the Riesz potential operator of order $-1$ with respect to all space variables. The proof of \eqref{Kato3d} follows the
same lines as in the $2$ $D$ case, the calculation is carried out by Ribaud and Vento in \cite[Proposition 3.1]{RV12}. On the other hand we have the following
Strichartz type estimates due to Linares and Saut.
\begin{lemma}\label{Strichartz3d}
 Let $\frac14 \le \frac{1}{p}<\frac27$ and $s=\frac{6}{p}-\frac32$. Then
\begin{equation}\label{StrichartzEst3d}
 \|I_x^sU_{\phi}u_0\|_{L^p_{xyt}}\ls \|u_0\|_{L^2_{xy}}.
\end{equation}
\end{lemma}
The derivative gain here involves only the $x$ - variable, not the full gradient. For $p<4$ this estimate is the special case of
\cite[Proposition 3.1]{LS}, where $p=q$. \footnote{The regularity gain in the $p=q$ - version written down here is restricted by
$s<\frac{3}{14}$, the nonsymmetric version is stronger and exhibits a gain of up to $\frac38 -$ derivatives, see \cite{LS}. For our
purposes an $I_x^{\e}$ will do, but this $\e$ is essential in our treatment of the quartic nonlinearity.} The case $p=q=4$, which will play a major role in our considerations, can be obtained by
similar arguments. An alternative approach (allowing a bilinear refinement) was sketched in Section 2 of \cite{G14}.

\quad

A problem seems to occur, if we try to prove an appropriate maximal function estimate (global in time and even without an $\e$
unnecessary derivative loss), since the symmetrization argument we applied successfully in $2$ $D$ fails in three space dimensions.
Nonetheless, let us for a short heuristic consider the symmetric phase function
$$\tilde{\phi}(\xi,\eta_1,\eta_2)=\xi^3+\eta_1^3+\eta_2^3.$$
Then the argument in the proof of Proposition \ref{max} gives the bound
$$ \|U_{\tilde{\phi}}u_0\|_{L^4_{xy}L_t^{\infty}}\ls \|I_x^{\frac14}I_y^{\frac12}u_0\|_{L^2_{xy}},$$
which shows, what we may expect: The loss of  $\frac34$  derivatives in an $L^4_{xy}L_t^{\infty}$ - estimate. It turns out that a fairly
soft argument combined with the Strichartz type estimate \eqref{StrichartzEst3d} will give us an appropriate substitute. This works,
since we are in three dimensions and the phase function is cubic.

\begin{lemma}
 Assume $\displaystyle 0<\frac{1}{q}\le\frac{1}{p}<\frac27$ and $\displaystyle \frac{1}{q}+\frac{5}{p}\le \frac32$. Then for
$\displaystyle s=3(\frac12-\frac{1}{p}-\frac{1}{q})$ we have
\begin{equation}\label{5.5}
 \|U_{\phi}u_0\|_{L^p_{xy}L_t^{q}}\ls\|u_0\|_{\dot{H}^s_{xy}}.
\end{equation}
\end{lemma}

\begin{proof}
 By Sobolev embedding in the space variables we may assume $p \le 4$. Let $u=U_{\phi}u_0$. Then for the space-time Fourier transform
of $u$ we have
$$\F u (\xi,\eta, \tau)= \delta_0 (\tau - \phi(\xi,\eta))\F_{xy}u_0(\xi,\eta),$$
so that $\tau=\phi(\xi,\eta)=\xi(\xi^2+|\eta|^2)$ in the support of $\F u$. Now if $p \le q < \infty$ we can apply a Sobolev embedding
in the time variable to obtain with $\displaystyle s_p=\frac{6}{p} - \frac{3}{2}$ as in Lemma \ref{Strichartz3d}
\begin{eqnarray*}
 \displaystyle \|U_{\phi}u_0\|_{L^p_{xy}L_t^{q}}\ls & \|\F^{-1}|\tau|^{\frac{1}{p}-\frac{1}{q}}\F U_{\phi}u_0\|_{L^p_{xyt}} 
  =  \|\F^{-1}|\xi(\xi^2+|\eta|^2)|^{\frac{1}{p}-\frac{1}{q}}\F U_{\phi}u_0\|_{L^p_{xyt}} \\
 \displaystyle \ls &  \hspace{-3,2cm}\|\F^{-1}_{xy}|\xi|^{\frac{1}{p}-\frac{1}{q}-s_p}(\xi^2+|\eta|^2)^{\frac{1}{p}-\frac{1}{q}}\F_{xy}u_0\|_{L^2_{xy}},
\end{eqnarray*}
where in the last step we have applied \eqref{StrichartzEst3d}. The assumption $\displaystyle \frac{1}{q}+\frac{5}{p}\le \frac32$
implies that $\displaystyle \frac{1}{p}-\frac{1}{q}-s_p \ge 0$, so that the Fourier multiplier can be estimated by $(\xi^2 + |\eta|^2)^{\frac{s}{2}}$.
\end{proof}

\subsection{The multilinear estimate on dyadic pieces in $3$ $D$} \hfill \\

Here we prove the estimate on dyadic pieces in three dimensions, which corresponds to Lemma \ref{multpieces2D} in Section 4.1.
This will look like a copy, but there are differences. We fix $s_c=\frac32-\frac{2}{k}$ for the remaining section and recall
that for the $3$ $D$ - case we have chosen the auxiliary quantity as
$$|P_N u|_{(k)}:= N^{s_c} \|I_x^{\frac{1}{10}}P_N u\|_{L^{\frac{15}{4}}_{xyt}}+N^{s_c}\|P_N u\|_{L^4_{xyt}}
+N^{\frac34-\frac{3}{2k}}\|P_N u\|_{L^4_{xy}L^{6k}_t}.$$
By the linear estimates \eqref{StrichartzEst3d} and \eqref{5.5} the three contributions are controlled by
$$|P_N u|_{(k)} \ls N^{s_c}\|P_N u\|_{V^2_{\phi}}.$$
Again, if the time intervall is taken $(0,T)$ in the involved norms, we write $|P_N u|_{(k,T)}$ instead of $|P_N u|_{(k)}$ and then
we can rely on $\lim_{T\to0}|P_N u|_{(k)}=0$ for all $u\in V^2_{\phi}$.

\begin{lemma}\label{multpieces3D}
 Let $u_1, \dots , u_{k+1} \in V^2_{\phi}$, $w \in U^2_{\phi}$ with $\|w\|_{U^2_{\phi}}\le 1$,
$N, N_1, \dots , N_{k+1}$ dyadic numbers with $N_1 \le N_2 \le \dots \le N_{k+1}$ and $N \ls N_{k+1}$. Then there exists 
$\e>0$ such that
$$ N^{s_c} \left| \int_{\R^3}P_{N_1}u_1 \cdot ... \cdot P_{N_{k+1}}u_{k+1} \cdot \partial_x P_Nw dxdydt \right| 
\ls  N^{\e}N_1^{\e}N_{k+1}^{-2\e} \prod_{j=1}^{k+1}|P_{N_j} u_j|_{(k)}.$$
\end{lemma}

\begin{proof}
 We consider two cases.
\begin{itemize}
 \item[Case 1:] $N_{k+1}^{\e}|\xi|\ls |\xi_{k+1}|^{\e}N$, \\

 \item[Case 2:] $N_{k+1}^{\e}|\xi|\ls N_k^{\e}N$.\\
\end{itemize}
In fact there is no further alternative. Clearly, we are in Case 2, if $N_{k+1}\ls N_k$. Otherwise we have $N_{k+1}\ls N$ and hence
$N_{k+1}^{\e}|\xi|^{1-\e}\ls N$. Now, since $\xi=\sum_{j=1}^{k+1}\xi_j$, we have
\begin{itemize}
 \item[(i)] $|\xi|\ls |\xi_{k+1}|$, hence $|\xi|^{\e}\ls |\xi_{k+1}|^{\e}$ and we are in Case 1, or
 \item[(ii)] $|\xi|\ls |\xi_{j}|$ for one $j \in \{1,...,k\}$, hence $|\xi|^{\e}\ls N_k^{\e}$, and we are in Case 2 again.
\end{itemize}

\quad

\emph{Estimation for Case 1:} We use $N_{k+1}^{\e}|\xi|\ls |\xi_{k+1}|^{\e}N$, $N\ls N_{k+1}$ and the Kato smoothing effect
to obtain
\begin{equation}\label{ub40}
N^{\e}N_{k+1}^{s_c-2\e} \|P_{N_1}u_1 \cdot ... \cdot P_{N_k}u_k \cdot (I_x^{\e} P_{N_{k+1}}u_{k+1})\|_{L_x^1L^2_{yt}}
\end{equation}
as upper bound for the contribution from this case. We choose H\"older exponents $p_j$ and $q_j$ with
$$ \qquad \frac{1}{p_1}=\frac{3}{4k}-\frac{\e}{6}, \qquad \qquad \frac{1}{q_1}=\frac{1}{4k}-\frac{\e}{6},\qquad$$
and, for $j\in\{2,\dots,k\}$,
$$ \qquad \frac{1}{p_j}=\frac{3}{4k}, \qquad \qquad \frac{1}{q_j}=\frac{1}{4k},\qquad$$
as well as
$$ \qquad \frac{1}{p_{k+1}}=\frac{1}{4}+\frac{\e}{6}, \qquad \qquad \frac{1}{q_{k+1}}=\frac{1}{4}+\frac{\e}{6},\qquad$$
so that $\displaystyle \sum_{j=1}^{k+1}\frac{1}{p_j}=1$ and $\displaystyle \sum_{j=1}^{k+1}\frac{1}{q_j}=\frac12$. H\"older's inequality gives
$$\|P_{N_1}u_1 \cdot ... \cdot P_{N_k}u_k \cdot (I_x^{\e} P_{N_{k+1}}u_{k+1})\|_{L_x^1L^2_{yt}} 
\le \Big(\prod_{j=1}^{k}\|P_{N_j}u_j\|_{L_x^{p_j}L_{yt}^{q_j}}\Big)\|I_x^{\e}P_{N_{k+1}}u_{k+1}\|_{L_{xyt}^{p_{k+1}}}.$$
For the first factor we use Sobolev embeddings in the space variables to obtain
$$\|P_{N_1}u_1\|_{L_x^{p_1}L_{yt}^{q_1}}=N_1^{\e}N_1^{-\e}\|P_{N_1}u_1\|_{L_x^{p_1}L_{yt}^{q_1}} \ls
N_1^{\e}N_1^{\frac34-\frac{5}{4k}-\frac{\e}{2}}\|P_{N_1}u_1\|_{L^4_{xy}L_t^{q_1}}.$$
Using a convexity inequality we can control $N_1^{\frac34-\frac{5}{4k}-\frac{\e}{2}}\|P_{N_1}u_1\|_{L^4_{xy}L_t^{q_1}}$ by the
second and third term in the auxiliary quantity $|\cdot |_{(k)}$ and we arrive at
$$\|P_{N_1}u_1\|_{L_x^{p_1}L_{yt}^{q_1}} \ls N_1^{\e}|P_{N_1}u_1 |_{(k)}.$$
In this calculation for $u_1$ we may take $\e=0$ and have for $j\in\{2,\dots,k\}$ the bound
$$\|P_{N_j}u_j\|_{L_x^{p_j}L_{yt}^{q_j}} \ls |P_{N_j}u_j |_{(k)}.$$
Finally for $u_{k+1}$ we have $\|I_x^{\e}u\|_{L_{xyt}^{p_{k+1}}}\ls \|I_x^{\frac{1}{10}} u\|_{L^{\frac{15}{4}}_{xyt}}+\| u\|_{L^4_{xyt}}$
and hence 
$$N_{k+1}^{s_c}\|I_x^{\e}P_{N_{k+1}}u_{k+1}\|_{L_{xyt}^{p_{k+1}}}\ls |P_{N_{k+1}}u_{k+1} |_{(k)}.$$
Summarizing we get
$$\eqref{ub40} \ls N^{\e}N_1^{\e}N_{k+1}^{-2\e} \prod_{j=1}^{k+1}|P_{N_j} u_j|_{(k)}.$$

\quad

\emph{Treatment of Case 2:} Here we apply $|\xi| \ls N_k^{\e}N_{k+1}^{-\e}N$, eliminate the $N$ by the application of the local
smoothing estimate and remain with the task of estimating
\begin{equation}\label{ub60}
N^{\e}N_k^{\e}N_{k+1}^{s_c-2\e} \|P_{N_1}u_1 \cdot ... \cdot  P_{N_{k+1}}u_{k+1}\|_{L_x^1L^2_{yt}}.
\end{equation}
We choose
$$ \frac{1}{p_1}=\frac{3}{4k}, \qquad \qquad \qquad \frac{1}{q_1}=\frac{1}{4k}-\frac{\e}{2}, \qquad \qquad \qquad \frac{1}{r_1}=\frac{1}{4k}, $$
for $j \in \{2, \dots , k-1\}$,
$$ \frac{1}{p_j}=\frac{3}{4k}, \qquad \qquad \qquad \frac{1}{q_j}=\frac{1}{4k}, \qquad \qquad \qquad \frac{1}{r_j}=\frac{1}{4k}, $$
as well as
$$ \frac{1}{p_k}=\frac{3}{4k}, \qquad \qquad \qquad \frac{1}{q_k}=\frac{1}{4k}+\frac{\e}{2}, \qquad \qquad \qquad \frac{1}{r_k}=\frac{1}{4k}, $$
and
$$ \frac{1}{p_{k+1}}=\frac{1}{4}, \qquad \qquad \qquad \frac{1}{q_{k+1}}=\frac{1}{4}, \qquad \qquad \qquad \frac{1}{r_{k+1}}=\frac{1}{4}. $$
H\"older's inequality gives
 $$ \|P_{N_1}u_1 \cdot ... \cdot  P_{N_{k+1}}u_{k+1})\|_{L_x^1L^2_{yt}} 
\le  \Big( \prod_{j=1}^{k}\|P_{N_j}u_j\|_{L_x^{p_j}L_{y}^{q_j}L_{t}^{r_j}}\Big)\|P_{N_{k+1}}u_{k+1}\|_{L^4_{xyt}}.$$
Sobolev inequalities in $x$ and $y$ give
$$\|P_{N_1}u_1\|_{L_x^{p_1}L_{y}^{q_1}L_{t}^{r_1}}\ls N_1^{\e}N_1^{\frac34-\frac{5}{4k}}\|P_{N_1}u_1\|_{L^4_{xy}L_t^{4k}}
 \ls N_1^{\e}|P_{N_1} u_1|_{(k)},$$
the latter by earlier calculation. Similarly we have for $j \in \{2, \dots , k-1\}$ that
$\|P_{N_j}u_j\|_{L_x^{p_j}L_{y}^{q_j}L_{t}^{r_j}}\ls |P_{N_j} u_j|_{(k)}$, and for the $k$th factor by almost the same
Sobolev embeddings
$$\|P_{N_k}u_k\|_{L_x^{p_k}L_{y}^{q_k}L_{t}^{r_k}}\ls N_k^{-\e}N_k^{\frac34-\frac{5}{4k}}\|P_{N_k}u_k\|_{L^4_{xy}L_t^{4k}}
\ls  N_k^{-\e}|P_{N_k} u_k|_{(k)}.$$
The estimate for $u_{k+1}$ is clear, since the $L^4_{xyt}$ - norm is a part of $|\cdot|_{(k)}$. Collecting terms we arrive at
$$\eqref{ub60} \ls N^{\e}N_1^{\e}N_{k+1}^{-2\e} \prod_{j=1}^{k+1}|P_{N_j} u_j|_{(k)},$$
which completes the calculation.
\end{proof}

The further procedure is now the same as for the symmetrized equation in $2$ $D$. From the quantities $|P_N u|_{(k)}$ and
$|P_N u|_{(k,T)}$, respectively, one builds the auxiliary norms $\| u\|_{(k,q)}$ and $\| u\|_{(k,q,T)}$ as norms of Besov type.
Since we avoided to use an $L_t^{\infty}$ - norm, we have $\lim_{T\to0}\| u\|_{(k,q,T)}=0$, whenever $u$ belongs to our solution
space. For $u_1, \dots , u_{k+1} \in \dot{X}^{s_c}_{q}$ one defines 
$$F(u_1, \dots , u_{k+1})(t):=\int_0^tU_{\phi}(t-s)\partial_x(u_1\cdot ... \cdot u_{k+1})(s)ds.$$
Summation of the dyadic pieces as in Lemma \ref{mult2D} gives $F(u_1, \dots , u_{k+1})\in \dot{X}^{s_c}_{q}$ and the estimate
$$\|F(u_1, \dots , u_{k+1})\|_{\dot{X}^{s_c}_{q}} \ls \prod_{j=1}^{k+1}\|u_j\|_{(k,q)},$$
which, if inserted into the proof of Theorem \ref{mainsym}, leads to the claimed local and global well-posedness result in $3$ $D$.
No further argument comes in, which is specific for the $3$ $D$ - case.


\begin{thebibliography}{999}

 \bibitem{BKS}Ben-Artzi, Matania; Koch, Herbert; Saut, Jean-Claude Dispersion estimates for third order
equations in two dimensions.  Comm. Partial Differential Equations  28  (2003),  no. 11-12, 1943--1974

 \bibitem{BL}Biagioni, H. A.; Linares, F. Well-posedness results for the modified Zakharov-Kuznetsov equation. 
Nonlinear equations: methods, models and applications (Bergamo, 2001),  181--189, Progr. Nonlinear Differential
Equations Appl., 54, Birkh\"auser, Basel, 2003

 \bibitem{BJM} Bustamante, Eddye; Jimenez Urrea, Jose; Mejia, Jorge; The Zakharov-Kuznetsov equation in weighted Sobolev spaces.
J. Math. Anal. Appl. 433 (2016), no. 1, 149--175

 \bibitem{CKZ13} Carbery, A.; Kenig, C. E.; Ziesler, S. N. Restriction for homogeneous polynomial surfaces in $\R^3$.
Trans. Amer. Math. Soc. 365 (2013), no. 5, 2367--2407

 \bibitem{F95}Faminskii, A. V. The Cauchy problem for the Zakharov-Kuznetsov equation. (Russian) 
Differentsialʹnye Uravneniya  31  (1995),  no. 6, 1070--1081, 1103;  translation in  Differential Equations  31
 (1995),  no. 6, 1002--1012

 \bibitem{F08} Faminskii, Andrei V. Well-posed initial-boundary value problems for the Zakharov-Kuznetsov equation.
Electron. J. Differential Equations 2008, No. 127, 23 pp.

 \bibitem{FLP}Farah, Luiz; Linares, Felipe, Pastor, Ademir A note on the 2D generalized Zakharov-Kuznetsov
equation: local, global, and scattering results. J. Differential Equations 253 (2012), no. 8, 2558--2571

 \bibitem{FP}Fonseca, German E., Pachon, Miguel A. Well-posedness for the two dimensional generalized Zakharov-Kuznetsov
equation in weighted Sobolev spaces, arXiv:1501.00220

 \bibitem{G05} Gr\"unrock, Axel A bilinear Airy-estimate with application to gKdV-3. Differential Integral Equations 18 (2005),
no. 12, 1333--1339

 \bibitem{G14} Gr\"unrock, Axel A remark on the modified Zakharov-Kuznetsov equation in three space dimensions.
Math. Res. Lett. 21 (2014), no. 1, 127--131

 \bibitem{GH} Gr\"unrock, Axel; Herr, Sebastian The Fourier restriction norm method for the Zakharov-Kuznetsov equation.
Discrete Contin. Dyn. Syst. 34 (2014), no. 5, 2061--2068. 

  \bibitem{HHK} Hadac, Martin; Herr, Sebastian; Koch, Herbert Well-posedness and scattering for the KP-II equation in a critical
space. Ann. Inst. H. Poincaré Anal. Non Linéaire 26 (2009), no. 3, 917--941

 \bibitem{KPV91}Kenig, Carlos E.; Ponce, Gustavo; Vega, Luis Oscillatory Integrals and Regularity of
Dispersive Equations, Indiana Univ. Math. J. 40 (1991), 33--69

 \bibitem{KPV93}Kenig, Carlos E.; Ponce, Gustavo; Vega, Luis Well-posedness and scattering results for
the generalized Korteweg-de Vries equation via the contraction principle, CPAM 46 (1993), 527--620

 \bibitem{Koch}Koch, Herbert Adapted function spaces for dispersive equations. Singular phenomena and scaling in mathematical
models, 49--67, Springer, Cham, 2014

 \bibitem{KoMa}Koch, Herbert; Marzuola, Jeremy L. Small data scattering and soliton stability in $\dot{H}^{-\frac16}$ for the
quartic KdV equation. Anal. PDE 5 (2012), no. 1, 145--198

 \bibitem{KT05}Koch, Herbert; Tataru, Daniel Dispersive estimates for principally normal pseudodifferential
operators. Comm. Pure Appl. Math. 58 (2005), no. 2, 217--284

 \bibitem{KT07}Koch, Herbert; Tataru, Daniel A priori bounds for the 1D cubic NLS in negative Sobolev spaces.
Int. Math. Res. Not. IMRN 2007, no. 16, Art. ID rnm053, 36 pp.

 \bibitem{LaSp}Laedke, E. W.; Spatschek, K. H. Nonlinear ion-acoustic waves in weak magnetic fields.
Phys. Fluids 25 (1982), no. 6, 985–989.

 \bibitem{LLS} Lannes, David; Linares, Felipe; Saut, Jean-Claude The Cauchy problem for the Euler-Poisson system and
derivation of the Zakharov-Kuznetsov equation. Studies in phase space analysis with applications to PDEs, 181--213,
Progr. Nonlinear Differential Equations Appl., 84, Birkh\"auser/Springer, New York, 2013

 \bibitem{LaTr}Larkin, Nikolai A.; Tronco, Eduardo Regular solutions of the 2D Zakharov-Kuznetsov equation on a half-strip.
J. Differential Equations 254 (2013), no. 1, 81--101

 \bibitem{LP09}Linares, Felipe; Pastor, Ademir Well-posedness for the two-dimensional modified
Zakharov-Kuznetsov equation.  SIAM J. Math. Anal.  41  (2009),  no. 4, 1323--1339

 \bibitem{LP11}Linares, Felipe; Pastor, Ademir Local and global well-posedness for the 2D generalized
Zakharov-Kuznetsov equation.  J. Funct. Anal.  260  (2011),  no. 4, 1060--1085. 

 \bibitem{LS}Linares, Felipe; Saut, Jean-Claude The Cauchy problem for the 3D Zakharov-Kuznetsov equation. 
Discrete Contin. Dyn. Syst.  24  (2009),  no. 2, 547--565

 \bibitem{LPS}Linares, Felipe; Pastor, Ademir; Saut, Jean-Claude Well-posedness for the ZK equation in a
cylinder and on the background of a KdV soliton.  Comm. Partial Differential Equations  35  (2010), 
no. 9, 1674--1689

 \bibitem{MP15}Molinet, Luc; Pilod, Didier Bilinear Strichartz estimates for the Zakharov-Kuznetsov equation and applications.
Ann. Inst. H. Poincaré Anal. Non Linéaire 32 (2015), no. 2, 347--371 

 \bibitem{MR03} Molinet, Luc; Ribaud, Francis On the Cauchy problem for the generalized Korteweg-de Vries equation.
Comm. Partial Differential Equations 28 (2003), no. 11-12, 2065--2091

 \bibitem{RV12} Ribaud, Francis; Vento, Stephane Well-posedness results for the three-dimensional
Zakharov-Kuznetsov equation. SIAM J. Math. Anal. 44 (2012), no. 4, 2289--2304

 \bibitem{RV12a}  Ribaud, Francis; Vento, Stephane A note on the Cauchy problem for the 2D generalized
Zakharov-Kuznetsov equations. C. R. Math. Acad. Sci. Paris 350 (2012), no. 9-10, 499--503

 \bibitem{STW} Saut, Jean-Claude; Temam, Roger; Wang, Chuntian An initial and boundary-value problem for the Zakharov-Kuznestov
equation in a bounded domain. J. Math. Phys. 53 (2012), no. 11, 115612, 29 pp.

 \bibitem{TT} Tao, Terence Scattering for the quartic generalised Korteweg-de Vries equation. J. Differential Equations 232 (2007), no. 2, 623--651

 \bibitem{Wiener}Wiener, Norbert The quadratic variation of a function and its Fourier coefficients, Massachusetts J. Math. 3 (1924), 72-94

\bibitem{ZK74}
V.E. Zakharov and E.A. Kuznetsov.
\newblock Three-dimensional solitons.
\newblock {\em Sov. Phys. JETP}, 39(2):285--286, 1974.
\end{thebibliography}
\end{document}